\documentclass[12pt]{amsart}
\usepackage{amssymb,graphicx,amsmath}
\usepackage[stable]{footmisc}
\newtheorem{thm}{Theorem}
\newtheorem{prop}{Proposition}

\newtheorem{lem}{Lemma}
\newtheorem{cor}{Corollary}
\newtheorem*{mainthm}{Main Theorem}

\theoremstyle{definition}
\newtheorem{defn}{Definition}

\newcommand{\f}[1]{\mathbb{#1}}
\newcommand{\lr}{\longrightarrow}

\def\om{\omega}
\def\Om{\Omega}
\newcommand{\m}[1]{\mathcal{#1}}
\makeatletter
\def\blfootnote{\xdef\@thefnmark{}\@footnotetext}
\makeatother

\usepackage[ marginparsep=5mm,heightrounded,centering,margin=3.47cm]{geometry}

\title[Universal groups of intermediate growth and their IRS]{Universal groups of intermediate growth and their invariant random subgroups}

\author{Mustafa G\"okhan Benli, Rostislav Grigorchuk and  Tatiana Nagnibeda}

\address{Mustafa G\"okhan Benli\\  Middle East Technical University, Ankara, Turkey}
\email{benli@metu.edu.tr}

\address{Rostislav Grigorchuk\\  Texas A\&M University, College Station, TX, USA }
\email{grigorch@math.tamu.edu}

\address{Tatiana Nagnibeda\\  Section de math\'ematiques \\
Universit\'e de Gen\`eve \\
C.P.~64 \\
\newline
CH--1211 Gen\`eve~4 \\ Suisse.}
\email{tatiana.smirnova-nagnibeda@unige.ch}

\date{June 26, 2015}

\begin{document}
 
 \maketitle

\begin{center}
\begin{Small}
 \textit{To A.M. Vershik on the occasion of his 80-th birthday, with admiration and respect}
 \end{Small}
\end{center}

\begin{abstract}
We exhibit examples of groups of  intermediate  
growth  with $2^{\aleph_0}$   ergodic continuous  invariant  random subgroups. The examples are the universal groups associated with a family of groups of intermediate growth.
\end{abstract}
 
\section{Introduction}

\blfootnote{The first and second authors were supported by NSF Grant DMS-1207699. The second and the third author were supported by the Swiss National Science Foundation.}

The  goal  of  this  paper is  to  show the existence of groups of  intermediate  
growth  with $2^{\aleph_0}$  distinct ergodic continuous  invariant  random subgroups.

Invariant  random subgroup (abbreviated  $IRS$) is a convenient term that stands for a probability measure on the space of subgroups in a locally compact group, invariant under the action of the group by conjugation.
In the case of a  countable   group  $G$ (only such groups will be considered here), the space $S(G)$
of subgroups of $G$ is supplied with the topology induced
from the Tychonoff topology on $\{0,1\}^G$ where
a subgroup $H\le G$ is identified with its characteristic function 
$\chi_H(g)=1$ if $g\in H$ and $0$ otherwise. 

The delta mass corresponding to a normal subgroup is 
a trivial example of an $IRS$, as well as the average over a finite orbit of delta 
masses associated with groups in a finite conjugacy class. Hence, we are rather
interested in continuous invariant probability
measures on $S(G)$. Clearly, such a measure does not necessarily exist, for example if the group only has countably many subgroups. 

Given a countable group $G$, a basic question is whether a continuous $IRS$ exists. Ultimately one wants to describe 
the structure of the simplex of invariant probability measures
of the topological dynamical system $(Inn(G),S(G))$ where $Inn(G)$
is the group of inner automorphisms of $G$ acting on $S(G)$. Of particular interest are ergodic measures, i.e., the extremal points in the simplex.

A more general problem is the identification of the simplex of invariant
probability measures of the system $(\Phi,S(G))$ where $\Phi$ is a subgroup
of the group $Aut(G)$ of automorphisms of $G$ (see \cite{1201.3399,bowen,MR2952431}).
 A closely related problem is the study of invariant measures on the space of rooted Schreier
graphs of $G$, with $G$ acting by change of the root.
This point of view is presented in \cite{MR2893544,MR2931920}.

A very fruitful idea in the subject belongs to Anatoly Vershik who introduced the notion of a  totally  non-free  
action  of a locally compact group $G$ on a  space $X$  with  invariant measure $\mu$,  i.e., 
an action  with the property  that  different  points  $x\in X$ have  different  stabilizers  
$St_G(x)$  $\mu$-almost  surely.  Then  the map $St: X \to  S(G)$ defined by $x\mapsto  St_G(x)$  
is  injective $\mu$-almost surely   
and  the  image of $\mu$ under this map is the law of an  $IRS$  on $G$  which  is  continuous  and ergodic  whenever  $\mu$ is.
In \cite{MR2952431}, Vershik showed that a totally non-free action of a group $G$ provides us not only with an $IRS$ but also with a factor representation of $G$. He also realized the plan outlined above and described all the ergodic $Aut(G)$-invariant measures on $S(G)$ in the case when $G$ is the infinite symmetric group, see \cite{Ver11,MR2952431}.

Bowen showed in \cite{bowen} that non-abelian free groups of finite rank possess a whole \lq\lq zoo\rq\rq\ of ergodic continuous $IRS$, and that a big part of the simplex of $IRS$ on a free group $F_r$, $r\geq 2$, is a Poulsen simplex.
(A simplex is called a Poulsen if its extremal points are dense. It is unique up to affine 
isomorphism by \cite{MR500918}). As shown
in \cite{BGK},  already the so-called lamplighter group $\m{L}=\f{Z}_2 \wr \f{Z}$
(the "simplest" finitely generated group that has $2^{\aleph_0}$ subgroups)
has a Poulsen simplex of $IRS$. 
Given a surjection $\phi: G \twoheadrightarrow H$, there is a natural homeomorphism
$\tilde{\phi}: S(H)\rightarrow S(G,Ker(\phi))$ where $S(G,Ker(\phi))$ denotes the 
subspace of $S(G)$ consisting of subgroups of $G$ containing $Ker(\phi)$. This allows to lift any 
$IRS$ on $H$ to $G$ thus providing a large spectrum of $IRS$ on $G$ from that on $H$. This applies in particular to the free group $F_2$ that covers $\m{L}$.

A finitely generated virtually nilpotent group has only countably
many subgroups and therefore does not possess any continuous $IRS$. By  Gromov's theorem,
the class of finitely generated virtually nilpotent groups coincide with the class
of groups of polynomial growth. Recall that, given a finitely generated group $G$ with 
a system of generators $S$, one can consider its growth function $\gamma(n)=\gamma_{(G,S)}(n)$
which counts the number of elements of length at most $n$. The growth type of this function when 
$n \to \infty$ does not depend on the generating set $S$ and can be polynomial, exponential or intermediate. The 
question of existence
of groups of intermediate growth was raised by Milnor  \cite{Milnor68}
and was answered by the second author in \cite{MR764305}. The main construction
associates with every sequence $\om \in \Om=\{0,1,2\}^{\f{N}}$ a group
$G_\om$ generated by four involutions $a_\om,b_\om,c_\om,d_\om$ and if $\om$ is not 
an eventually constant sequence, then $G_\om$ has intermediate growth. Moreover, it was also
observed in \cite{MR764305} that the groups $G_\om$ fall into the class of just-infinite
branch groups. A  group  is  just  infinite if  
it  is  infinite  but every proper  quotient   is  finite.
A  group  is branch  if  it  has  a  faithful  level transitive action on a  spherically  
homogeneous rooted  tree  with  the  property  that rigid  stabilizers  of the levels of the tree   are  of  
finite  index, see Section 4.2 for precise definitions.
Just-infinite branch  groups constitute one  of  three  
classes  in which  the  class of  just-infinite groups  naturally  
splits  \cite{MR1765119}.  

Since the groups $G_\om$ are just-infinite, they only have countably
many quotients. This raised the question of existence of groups of intermediate
growth having $2^{\aleph_0}$ quotients, answered 
in \cite{MR781246}. The main idea was to take a suitable subset $\Lambda \subset \Omega$
of cardinality $2^{\aleph_0}$ and consider the group $U_\Lambda$ defined as the 
quotient of the free group $F_4$ by a normal subgroup $N$ which is the intersection of 
normal subgroups $N_\om,\om\in\Lambda$ where $G_\om =F_4/N_\om$. 
In  this  paper  we  explore  this  idea  further by  using  $IRS$ on $G_{\omega}$   
and  lift  them  to  $U_{\Lambda}$   deducing  the main  result. 
 
Branch groups give us the most transparent examples of  totally non-free actions and thus of $IRS$.
Indeed, as shown in \cite{MR1899368}, the natural extension of a branch action on a spherically homogeneous tree $T$ to its boundary $\partial T$ is totally non free with respect to the uniform probability measure on $\partial T$.  It is even completely non free, i.e.,  different  points  have different  stabilizers. The  uniform probability measure on $\partial T$ is ergodic and invariant.  The  groups  $G_{\omega}$   act on the binary  rooted  
tree  in  a  branch  way.  By lifting  the uniform measure  to  $S(G_{\omega})$   and  then   to  $S(U_{\Lambda})$, one obtains a host of $IRS$ on $U_\Lambda$. We then proceed to showing that the $IRS$ obtained in this way are distinct. These considerations allow us to prove our main theorem:

\begin{mainthm}
 There exists a finitely generated group of intermediate growth with $2^{\aleph_0}$ distinct continuous ergodic invariant random subgroups.
\end{mainthm}

We also  investigate  some  additional  properties  of  
 groups of the form $ U_{\Lambda},\;\Lambda\subset \Omega$,  including  
finite presentability,  branching  property and self-similarity.

\section{Space of Marked Groups and Universal groups}

\begin{defn}
  A \textit{k-marked group} is a pair $(G,S)$, where $G$ is a group and $S=(s_1,\ldots,s_k)$ is
  an ordered set of (not necessarily distinct) elements such that the set $\{s_1,
  \ldots,s_k\}$ generates the group $G$. The \textit{canonical map} between two $k-$ marked groups  
  $(G,S)$ and $(H,T)$
  is the map sending $s_i \mapsto t_i$ $i=1,2,\ldots,k$.
  If this map defines an epimorphism, it will be called the marked epimorphism and $(H,T)$
  will be called a marked image of $(G,S)$.
  Two $k$-marked groups $(G,S)$ and $(H,T)$ are equivalent
  if the canonical map  defines  an isomorphism between $G$ and $H$.
\end{defn}

The space of (equivalence classes of) $k$-marked groups will be denoted by $\m{M}_k$. 
This space has a natural topology, which for instance can be defined by the following  metric: Two $k$-marked 
groups $(G,S)$ and $(H,K)$ are of distance $2^{-m}$, where $m$ is the largest natural number such that the balls of 
radius $m$ of the 
Cayley graphs of $(G,S)$ and $(H,K)$ are isomorphic (as directed labeled graphs). In \cite{MR764305}
it was observed that this makes $\m{M}_k$
into a compact totally disconnected space.

Alternatively, this space can be defined in the following way: 
Let $F_k$ be a free group of rank $k$ with a basis $\{x_1,\ldots,x_k\}$.
Let $\m{N}_k$ denote the set of all normal subgroups of  $F_k$,  together with the topology
inherited from the power set $\m{P}(F_k)\cong \{0,1\}^{F_k}$ supplied with the Tychonoff topology. 
This topology has basis consisting of  sets of the form 
$\m{O}_{A,B}=\{N\vartriangleleft F_k \mid A\subset N\;,\; B\cap N= \varnothing \}$ where $A$ and $B$
are finite subsets of $F_k$. 
Given $(G,S)\in \m{M}_k$,  let $N_{(G,S)} \in \m{N}_k$ be the kernel of 
the natural map $\pi_{(G,S)}:F_k\rightarrow G$ sending $x_i\mapsto s_i$. 
This gives a homeomorphism between $\m{M}_k$ and $\m{N}_k$ (depending on the basis of $F_k$)
(See \cite{MR1760424}).
We will interchangeably use these two spaces.

\begin{defn}\label{defnofuniversalgp}
 Let $\m{C}=\{(G_i,S_i) \mid i\in I\}$ be a subset of $\m{M}_k$. Let $ N_{\m{C}}=\bigcap_{i\in I}N_{(G_i,S_i)}$.
 The \textit{universal group of the family} $\m{C}$ 
 is the $k$-marked group $(U_\m{C},S_\m{C})$ where $U_\m{C}=F_k / N_{\m{C}}$ and $S_\m{C}$ is the image of the basis  
 $\{x_1,\ldots,x_k\}$.
\end{defn}

 $U_\m{C}$ has the following universal property:
If $(H,T)$ is a marked group such that for all $i \in I$ the canonical map from $(H,T)$ to 
$(G_i,S_i)$ defines a group homomorphism, then the canonical map from $(H,T)$ to $(U_\m{C},S_\m{C})$
defines a group homomorphism.
 
 An alternative way to define the universal group is the following: 
 
 \begin{defn}\label{diaggp}
Given 
 $\m{C}=\{(G_i,S_i) \mid i\in I\}\subset \m{M}_k$, write $S_i=(s^i_1,\ldots,s^i_k)$. 
 Let $U^{diag}_{\m{C}}$ be the subgroup of the (unrestricted) direct product $\prod_{i\in I} G_i $ 
 generated
 by the elements $s_j=(s^i_j)_{i\in I}\;j=1,\ldots,k$. 
 The $k$-marked group $(U^{diag}_{\m{C}},S^{diag}_{\m{C}})$ is called the \emph{diagonal group}
 of the family $\m{C}$.
  \end{defn}
  
  It is straightforward to check that $(U^{diag}_{\m{C}},S^{diag}_{\m{C}})$
  equivalent (as a marked group) to the universal group $(U_\m{C},S_\m{C})$ of Definition \ref{defnofuniversalgp}.

 \begin{prop}\label{closure}
Let $\m{C}\subset \m{M}_k$. Then the marked groups $(U_\m{C},S_\m{C})$ and 
$(U_{\m{\overline C}},S_{\m{\overline C}})$ are equivalent, where ${\m{\overline C}}$ denotes the closure of $\m{C}$
in $\m{M}_k$.
 \end{prop}

 \begin{proof}
  We need to show that $$ \bigcap_{(G,S)\in \m{C}} N_{(G,S)}= \bigcap_{(G,S)\in \m{\overline C}} N_{(G,S)}. $$
  Clearly, the right hand side is contained in the left. Suppose that some $g\in F_k$ belongs to the left
hand side but not to the right. Then there exists $(G,S)\in \m{C}$ such that $g \notin N_{(G,S)}$. Let 
$\{(G_n,S_n)\}_{n\ge 0}$ be a sequence in $\m{C}$ converging to $(G,S)$. Since $g$ belongs to the left hand side,
$g$ belongs to each $N_{(G_n,S_n)}$ and by definition of the topology in $\m{N}_k$, to $N_{(G,S)}$ which gives a 
contradiction.
 \end{proof}

 For an element $w\in F_k,w\neq 1$, denote $\m{O}_w=\{N \vartriangleleft F_k \mid w\in N\}$.
 
 \begin{lem}\label{cover}
   Let $H\le F_k$ be a subgroup and  $w_1,\ldots,w_m \in H,\;w_i\neq 1$ for each $i$.
   Then there exists  $w\in H,w\neq1$
   such that $\bigcup_{i=1}^m \m{O}_{w_i}\subset \m{O}_w$.
 \end{lem}
 
 \begin{proof}
  By induction on $m$. The case $m=1$ is clear, one can take $w=w_1$. So, assume $m>1$. 
  
  \textit{Case 1:} $[w_1,w_2]=1$ in $F_k$. In this case there exists 
  $w\in F_k$  and $s,t \in \mathbb{Z}$ such that $w_1^s=w_2^t=w$ (all non-trivial abelian subgroups in a free group are cyclic).
  Therefore, $\m{O}_{w_1} \cup \m{O}_{w_2} \subset \m{O}_w$ and hence we can apply the induction hypothesis by 
  replacing $\m{O}_{w_1}$ and $\m{O}_{w_2}$ by $\m{O}_w$.
  
  \textit{Case 2:} $[w_1,w_2]\neq 1$ in $F_k$. In this case we can replace $\m{O}_{w_1}$ and $\m{O}_{w_2}$
  by $\m{O}_{[w_1,w_2]}$ and apply the induction hypothesis.

 \end{proof}

 \begin{prop}\label{nf}
  Let $\m{C}\subset \m{M}_k$ be a closed subset and assume that no group in $\m{C}$ contains a nonabelian free subgroup.
  Then the universal group $U_{\m{C}}$ also has no nonabelian free subgroups.
 \end{prop}
 
 \begin{proof}Let $\m{C}=\{(G_i,S_i) \mid i\in I\}$.
  Let $a,b \in U_{\m{C}}$ be two distinct elements, given as words in the generators $S_C$.
  Let $w_a,w_b \in F_k$ such that $\pi_{(G,S)}(w_a)=a$ and $\pi_{(G,S)}(w_b)=b$.
  For each $i\in I$, since 
  $G_i$ has no (non-abelian) free subgroups, there is nontrivial $w_i \in \left< w_a,w_b\right> \le F_k$ such that   
  $\pi_{(G_i,S_i)}(w_i)=1$, i.e., $w_i \in N_{(G_i,S_i)}$. Hence $\{\m{O}_{w_i}\}_{i\in I}$ 
  is an open cover of $\m{C}$. Since $\m{C}$ is compact, there is  a finite 
  subcover  $\m{O}_{w_1},\ldots,\m{O}_{w_n}$.  By Lemma \ref{cover}, there exists non-trivial $w\in \left< w_a,w_b\right>$ such that 
  $\m{C}\subset \m{O}_w$. This shows that $w=1$ in $U_{\m{C}}$.
 \end{proof}

 \section{Grigorchuk 2-groups\footnote{
 The first and the third authors insist on using this standard terminology.}}\label{grig2gps}

 We recall  here the construction of \cite{MR764305}. Note that in the 
 original construction in \cite{MR764305} the groups are defined as measure 
 preserving transformations of the unit interval. Here we will 
 define them  as groups of automorphisms of the binary rooted tree.

Let $\Omega=\{0,1,2\}^{\f{N}}$ be the space of infinite sequences $\om=\om_1\om_2\ldots\om_n\ldots$ where 
$\omega_i\in \{0,1,2\}$,
considered with its natural product topology.  Let $\tau$ be the shift transformation, i.e., if $\om=\om_1\om_2\ldots \in \Om$ then
$\tau\om=\om_2\om_3\ldots$. Let $T$ be the  binary rooted  tree whose vertices are identified with the set of all
finite binary words $\{0,1\}^*$ and edges defined in standard way: $E=\{(w,wx) \mid w\in\{0,1\}^\ast,x\in\{0,1\} \}$.
For  each $\om \in \Om$, consider the  automorphisms $\{a,b_\om,c_\om,d_\om\}$ of $T$ defined
recursively as follows:

For  $v\in\{0,1\}^*$
$$a(0v)=1v \;\text{and} \; a(1v)=0v  $$

$$\begin{array}{llllll}
 b_\om(0v)=& 0 \beta(\om_1)(v) &   c_\om(0v)=& 0 \zeta(\om_1)(v)
  &  d_\om(0v)= &0 \delta(\om_1)(v) \\

   b_\om(1v)=& 1 b_{\tau \om}(v) &      c_\om(1v)= & 1 c_{\tau \om}(v)
   & d_\om(1v)= & 1 d_{\tau \om}(v), \\

\end{array}
$$
where
$$
\begin{array}{ccc}
 \beta(0)=a & \zeta(0)=a & \delta(0)=e \\
  \beta(1)=a & \zeta(1)=e & \delta(1)=a \\
   \beta(2)=e & \zeta(2)=a & \delta(2)=a \\
\end{array}
$$
and $e$ denotes the identity automorphism of $T$. 

For each $\om\in \Om$, let $G_\om$ be the subgroup of $Aut(T)$ generated by the set $S_\om=\{a,b_\om,c_\om,d_\om\}$
so that  $\m{G}=\{(G_\om,S_\om)\mid \om \in \Om\}$ is a subset of $\m{M}_4$. In \cite{MR764305} it was
observed that if two sequences $\om,\eta \in \Om$ which are not eventually constant, have
long common beginning, then the $4$-marked groups $(G_\om,S_\om)$ and $(G_\eta,S_\eta)$ 
are close to each other in $\m{M}_4$. It was also observed that the groups $(G_\om,S_\om)$ for 
eventually constant sequences $\om$ are isolated in $\{(G_\om,S_\om)\mid \om \in \Om\}$.
Hence, removing these isolated points from this set and taking its closure in $\m{M}_4$, one obtains
a compact subset $\m{\tilde G}=\{(\tilde G_\om,\tilde S_\om)\mid \om \in \Om\}\subset \m{M}_4$ which is homeomorphic
to $\Omega$ (and hence to a Cantor set) via $\om \mapsto (\tilde G_\om, \tilde S_\om)$. Note that $(\tilde G_\om, \tilde S_\om)=(G_\om,S_\om)$
if and only if $\om$ is not eventually constant and $(\tilde G_\om,\tilde S_\om)=\lim_{n\to \infty} (G_{\om^{(n)}},S_{\om^{(n)}})$
when $\om$ is eventually constant and where $\{\om^{(n)}\}_{n\ge 0}$ is a sequence of not eventually constant elements in $\Om$ converging to 
 $\om$ (the limit does not depend on the choice of the sequence $\{\om^{(n)}\}_{n\ge 0}$ ). 
In other words, the families $\m{G}$ and $\m{\tilde G}$ differ only on countably many points.

Note that we have the following:  $N_{(\tilde G_\om,\tilde S_\om)}=N_{(G_\om,S_\om)}$ if $\om$ is not eventually constant
 and $N_{(\tilde G_\om,\tilde S_\om)}\subset N_{(G_\om,S_\om)}$ for  eventually constant $\om \in \Om$.

Let $\Omega_\infty$ be the set of sequences in $\Omega$ in which all three letters 
$\{0,1,2\}$ occur infinitely often and $\Omega_0$ be the set of 
eventually constant sequences. Regarding the groups in $\m{G}$ and $\m{\tilde G}$ the following are known:

\begin{thm}[\cite{MR764305}]\label{grig84} \mbox{}

\begin{enumerate}
\item All groups $G_{\omega},\;\om\in\Om$ are infinite residually finite groups.
 \item $G_{\omega}$ is virtually $\f{Z}^{2^n}$ if $\omega$ is becomes constant starting with $n$-th coordinate.
\item If $\om \notin \Omega_0$ then $G_{\omega}$ has intermediate growth between polynomial
and exponential.
\item If $\om \in \Om_0$ then $\tilde G_\om$ is virtually metabelian, infinitely presented
and has exponential growth.
\item If $\om \in \Omega_\infty$ then $G_{\omega}$ is a torsion 2-group.
\item If $\om \in \Omega_\infty$ then $G_{\omega}$ is just-infinite, i.e., all its nontrivial quotients are finite.
\item For $\om_1,\om_2 \in \Om_\infty$ we have $G_{\om_1}\cong G_{\om_2}$ if and only if $\om_1$
can be obtained from $\om_2$ by applying a permutation from $Sym(\{0,1,2\})$ letter by letter.
\end{enumerate}
\end{thm}

\begin{proof}
 For proofs of (1),(2),(3) and (5) see \cite[Theorem 2.1]{MR764305}. (4) is proven in
 \cite[Theorem 6.1,6.2]{MR764305} and (6) in \cite[Theorem 8.1]{MR764305}. (7) is proven 
 in {\cite[Theorem 2.10.13]{MR2162164}}.
\end{proof}

\section{Some properties of the full universal group  $U$ } \label{fullunivgp}

Regarding the universal groups corresponding to the families $\m{G}$ and $\m{\tilde G}$ we have the following:

\begin{prop}
 $U_{\m{G}}=U_{\m{\tilde G}} $.
\end{prop}

\begin{proof}
Referring to the notation of Definition \ref{defnofuniversalgp}, we need to show the following equality: 
$$N_{\m{G}}:=\bigcap_{\om \in \Om} N_{(G_\om,S_\om)}=N_{\m{\tilde G}}=:\bigcap_{\om \in \Om} N_{(\tilde{G}_\om,\tilde{S}_\om)}.$$
Since $N_{(\tilde G_\om,\tilde S_\om)}\subset N_{(G_\om,S_\om)}$ for each $\om \in \Om$,  the right-hand side of the above equation is contained in the left-hand side.
Since  $\{(\tilde G_\om,\tilde S_\om)\mid \om \in \Om\setminus \Om_0\}$ is dense in $\m{\tilde G}$, by 
Proposition \ref{closure} we have
$$N_{\m{\tilde G}}=\bigcap_{\om \in \Om\setminus\Om_0} N_{(\tilde{G}_\om,\tilde{S}_\om)}.$$ Therefore,
$$N_{\m{G}}\subset  \bigcap_{\om \in \Om\setminus\Om_0} N_{({G}_\om,{S}_\om)}=
\bigcap_{\om \in \Om\setminus\Om_0} N_{(\tilde{G}_\om,\tilde{S}_\om)}=N_{\m{\tilde G}}.$$

\end{proof}


We will use the notation $U=U_{\m{G}}$ for the full universal
group and denote by $S=\{a,b,c,d\}$ its canonical generators. Note that the basic relations 
$a^2=b^2=c^2=d^2=bcd=1$ hold in $U$.

\begin{thm}
  $U$ contains no nonabelian free subgroups, has uniformly exponential growth and is not finitely
 presented.
\end{thm}

\begin{proof}
 Since all groups in $\m{\tilde G}$ are amenable (and hence cannot contain nonabelian free subgroups), the first assertion follows from Proposition \ref{nf}.
 By Theorem \ref{grig84} part (4), the group ${\tilde G}_\eta$ for $\eta=000\ldots$ 
 is an elementary amenable group of exponential growth, and hence of uniformly exponential growth
 by \cite{MR2110759}. Therefore $U$ has uniformly exponential growth.
 By \cite[Theorem 1.10]{MR3061134}, any finitely presented group mapping onto the groups $G_\om,\om\in\Om$
 must be large, i.e., has a finite index subgroup mapping onto a nonabelian free group. 
 In particular, such group contains a nonabelian free subgroup. Therefore $U$ 
 cannot be  finitely presented.
\end{proof}


\subsection{$U$ as an automaton group} \mbox{}

\smallskip
In this section we will realize $U$ as an automaton group and explore further properties. Firstly, 
we will recall some basics.
 
 Let $T_d$ denote the $d$-ary rooted tree with vertex set $\{0,1,2,\ldots,d-1\}^\ast$.
 For an automorphism $g\in Aut(T_d)$ and $x\in\{0,1,\ldots,d-1\}$, the \emph{section} of 
 $g$ at $x$
 (denoted by $g_x$) is the automorphism of $T_d$ defined uniquely by 
 $$g(xv)=g(x)g_x(v)\;\;\mbox{for all}\;\;v\in \{0,1,\ldots,d-1\} .$$
 This gives an isomorphism
$$\begin{array}{cccc}
   & Aut(T_d) & \rightarrow & S_d  \ltimes \left(Aut(T_d) \times \cdots \times Aut(T_d)\right ) \\
         &   g      &  \mapsto    & \sigma_g  (g_0,\ldots,g_{d-1})  
  \end{array}
$$
where $\sigma_g$ describes how $g$ permutes the first level subtrees and $g_i$ describe its action
within each subtree. (Here $S_d$ is the symmetric group on $d$ letters).

 \begin{defn}
  A subgroup $G\le Aut(T_d)$ is called self-similar if for all $g\in G$ and $x\in\{0,1,\ldots,d-1\}$,
  $g_x\in G$.
 \end{defn}
For an overview of self-similar groups and related topics we refer to \cite{MR2327317}.

A standard way to construct self-similar groups is to start with a list of symbols
$S=\{s^1,\ldots,s^m\}$ and permutations $\sigma_1,\ldots,\sigma_m \in S_d$ and consider the system
$$\begin{array}{ccc}

s^1&=&\sigma_1 (s^1_0,\ldots,s^1_{d-1}) \\
\vdots & \vdots & \vdots \\
s^m&=&\sigma_m (s^m_0,\ldots,s^m_{d-1} )\
  \end{array}
$$
where $s^i_j \in S $. If $\sigma_i = id$, we will omit writing it. Such a system defines a unique set of  $m$ automorphisms of $T_d$. 
Clearly the group $G=\left<S\right>$ will be self-similar.
Since in this case the generating set $S$ is closed under taking sections, the action of the
group can be modeled by a Mealy type automaton where each generator will correspond to a 
state of the automaton (see the figure below for an example).
Such groups, i.e., groups generated by the states
of a Mealy type automaton are called \emph{automata groups}. We refer to \cite{MR1841755} for a detailed account on automata groups.

 Consider the tree $T_6$ determined by the alphabet 
 $$\m{A}=\{0,1\} \times \{0,1,2\}=\{(0,0),(0,1),(0,2),(1,0),(1,1),(1,2)\}$$
 whose elements are enumerated as $0,1,\ldots,5$.
 Let $V\le Aut(T_6)$ be generated by the following elements:

 \begin{equation}\label{wreath}
\begin{array}{cccc}
A & = & (14)(25)(36) &(E,E,E,E,E,E)  \\
B & = & &(A,A,E,B,B,B)  \\
C & = & &(A,E,A,C,C,C)  \\
D & = & &(E,A,A,D,D,D) \\
 \end{array}
\end{equation} 
where, $(14)(25)(36)$ is an element of the symmetric group $S_6$ and $E$ corresponds to the
identity automorphism. Observe that $A^2=B^2=C^2=D^2=BCD=1$.
The corresponding automaton is as follows:
\begin{center}
\includegraphics[scale=0.54]{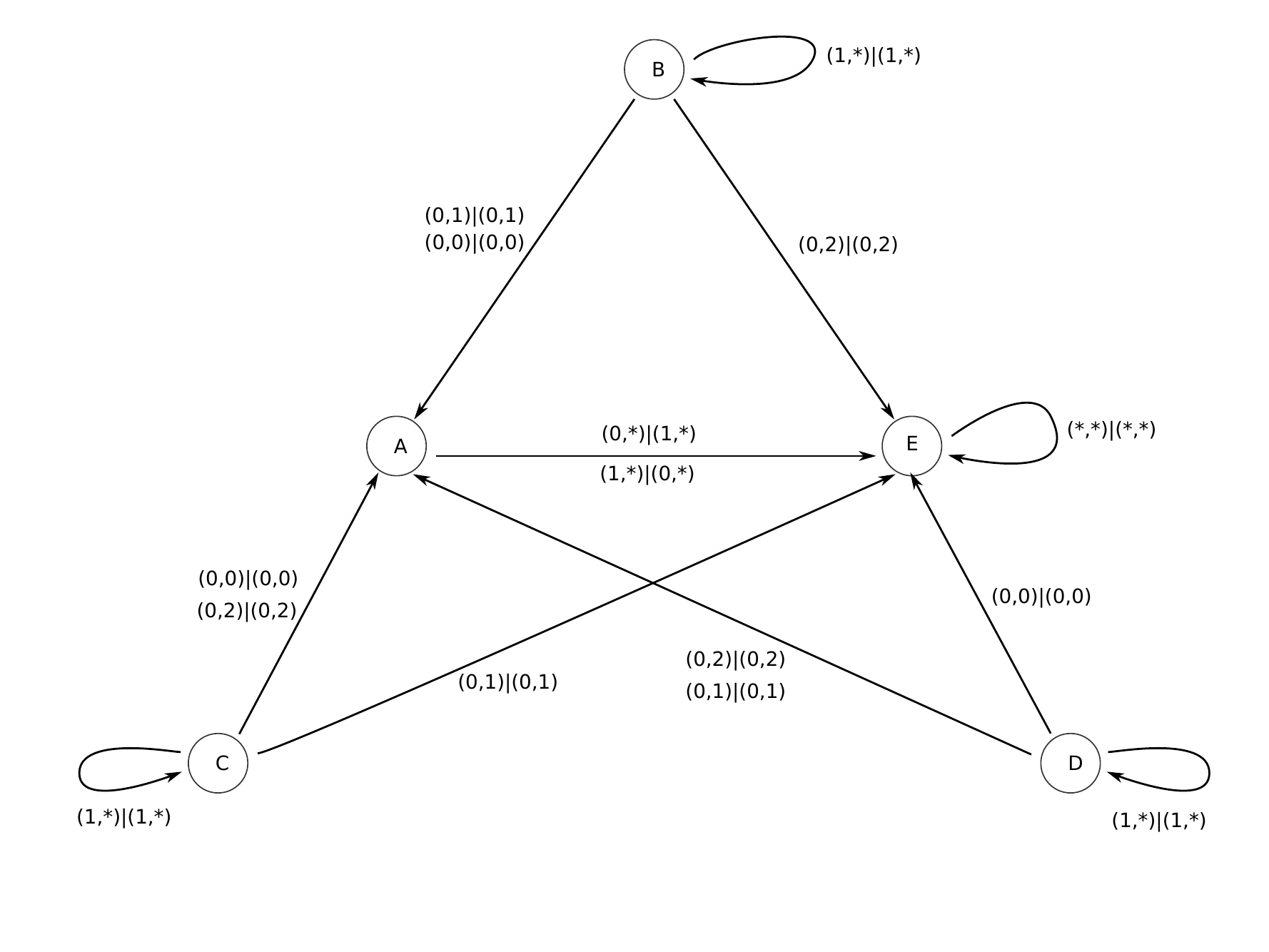} 
\end{center}

We will show that the group $V$ is isomorphic to $U$ (as a marked group).

Given $\om \in \Om$ and $u\in\{0,1\}^\ast$ let 
$\om^u\in \{0,1,2\}^\ast$ be the beginning of $\om$ of length $|u|$. Note that 
$$\om^{uv}=\om^u  \left(\tau^{|u|}\om\right)^v$$
for all $u,v\in\{0,1\}^\ast$

For any $\om\in\Om$ let $T_\om=\{(u,v)\in T_6 \mid u\in \{0,1\}^\ast\;,v=\om^u\}$. 
Clearly $T_\om$ is a binary 
subtree of $T_6$. Denote $\{0,1\}^\ast$ by $T_2$ and let $\phi_\om: T_\om \to T_2$ be  defined as 
$$\phi_\omega(u,v)=u$$ which clearly is
a bijection. For $(u,v)\in T_\om$ and $(u',v')\in T_{\tau^{|u|}\om}$ we have

\begin{equation}\label{sections}
 \phi_\om\left(uu',vv'\right)=\phi_\om(u,v)\phi_{\tau^{|u|}\om}(u',v') 
\end{equation}

\medskip

Given $g\in V$ and $\om  \in \Om$ define a group homomorphism $\psi_\om :V \to Aut(T_2) $ by 
$$\psi_\om(g)(u)=\phi_\om(g(u,\om^u))\;\;\text{for all}\;\; u \in T_2$$
It is straightforward to verify that $\psi_\om(g)\in Aut(T_2)$ and that $\psi_\om$ defines 
a group homomorphism.

\begin{lem}\label{sections_lem}
 For all $u\in T_2$ with $|u|=n$ 
 have $$\psi_\om(g)_u=\psi_{\tau^n\om}\left(g_{(u,\om^u)}\right) $$
\end{lem}

\begin{proof}Let $u,z\in\{0,1\}^\ast,|u|=n$ and denote $\om^u=v,(\tau^{n}\om)^z=v'$
 $$g(uz,\om^{uz})=g(uz,vv')=g\left((u,v)(z,v')\right)=g(u,v)g_{(u,v)}(z,v') $$
 Hence, by  Equation \ref{sections}, 
 $$\psi_\om(g)(uz)=\phi_\om(g(uz,\om^{uz}))=\phi_\om\left(g(u,v)g_{(u,v)}(z,v') \right)=
 \phi_\om(g(u,v)) \phi_{\tau^n\om}(g_{(u,v)}(z,v'))$$
 $$=\psi_\om(g)(u)\psi_{\tau^n\om}(g_{(u,v)})(z) $$
The result follows.
 \end{proof}
 
 \begin{lem} \label{mapping}
For any $\om \in \Om$,  $\psi_\om$ defines a marked surjective homomorphism
$\psi_\om:V \to G_\om$.

\end{lem}

\begin{proof}
It is enough to show that $\psi_\om$ maps generators of $V$ to the generators
of $G_\om$. Firstly, by definition of $A$ we have
$$\psi_\om(A)(u)=\phi_\om(A(u,\om^u))=\phi_\om((a(u),\om^u))=a(u)\;\text{for all}\;\;u.$$

We will show by induction on $|u|$ that $B,C,D$ are mapped to $b_\om,c_\om,d_\om$ respectively.
If $|u|=1$ it is straightforward to check this.  Using Lemma \ref{sections_lem} and induction
assumption we have for $u\in\{0,1\}^\ast$

$$\psi_\om(B)(0u)=0\psi_\om(B)_0(u)=0\psi_\om(B_{(0,\om^0)})(u) =\left\{
\begin{array}{cc}
 0a(u) & \;\; \text{if} \; \om^0=0,1 \\
 0u & \;\; \text{if} \; \om^0=2 
\end{array}\right.=b_\om(0u)
$$

Similarly one can check that $\psi_\om(B)(1u)=b_\om(1u)$ for all $u\in\{0,1\}^\ast$
and hence $\psi_\om(B)=b_\om$. 
Repeating the argument shows that $\psi_\om(C)=c_\om,\psi_\om(D)=d_\om$.

\end{proof}

\begin{thm}
 The group $V$ is isomorphic to the universal group $U$ (as a marked group).
\end{thm}

 \begin{proof}
  By Lemma \ref{mapping}, for each $\om \in \Om$ there exists a marked surjection 
  $\psi_\om: V\to G_\om$, and hence there exists a marked surjection $\psi:V\to U$.
  If $g\in V$ is a nontrivial, let $v\in T_6$ such that $gv\neq v$. Let $\om\in \Om$
  be such that $v\in T_\om$. This shows that $\psi_\om(g)\neq 1$ and hence $\psi(g)\neq 1$.
  This shows that $\psi$ is a marked isomorphism. 
 \end{proof}

From now on we will identify $U$ with $V$.


Sidki in \cite{MR1774362} classified automata groups according to their "activity growth" and conjectured that automata groups of polynomial growth are amenable. Note that the automaton defining $U$ has exponential activity growth in Sidki's classification. The question of amenability of the group $U$ remains open. The note \cite{Muchnik_note} claiming amenability of this group unfortunately contains a mistake.

\subsection{Branch Structure of $U$} \mbox{}

\smallskip

Let $G$ be a group acting on a rooted  $d$-ary tree $T_d$. For a vertex $v$ of $T_d$, let $T_v$ denote
the subtree hanging down at vertex $v$ and for an element $g\in G$ let $supp(g)$ be the support of $g$ i.e., 
the set of vertices not fixed by $g$. The stabilizer of a vertex $v$ is the subgroup
$St_G(v)=\{g\in G\mid g(v)=v\}$.
The rigid stabilizer of a vertex $v$ is the 
subgroup $Rist_G(v)=\{g \in G \mid supp(g) \subset T_v\}$. The rigid stabilizer of level $n$ 
is the subgroup $Rist_G(n)=\left<Rist_G(v) \mid |v|=n\right>$. Since
rigid stabilizer of distinct vertices of the same level commute, we have $Rist_G(n)=\prod_{|v|=n}Rist_G(v)$.
\begin{defn}
 Let $G$ be group of automorphisms of a rooted tree $T$. 
 $G$ is said to be a near branch group (resp. weakly near branch group)
 if for all $n\ge1$, the subgroup $Rist_G(n)$ has finite index in $G$ (resp. is nontrivial). 
 If in addition $G$ acts level transitively (i.e., transitively on each level of the tree) 
 then $G$  is called a branch group (weakly branch group) respectively.
\end{defn}

The class of (weakly) branch groups is interesting from various points of view
and plays an important role in the classification of just-infinite groups, i.e., 
infinite groups whose proper homomorphic images are all finite (see \cite{MR1765119}
for a detailed account on branch groups and just-infinite groups).

Let us mention the following fact which will be used in the forthcoming sections.
We will also give an alternative proof of this fact later.

\begin{thm}\cite{MR764305}
 For $\om \in \Om_\infty$, the group $G_\om$ is a branch group.
\end{thm}
Note that at the terminology ``branch group''  was not used in \cite{MR764305}.

If $G$ is a self-similar group, a standard way to show  near branch property (resp. weakly near branch property) is to find a
finite index subgroup $K$ (resp. nontrivial subgroup) of $G$ such that the image 
$\phi(K)$ contains the subgroup $K\times \cdots \times K$ where 
 $\phi :Aut(T_d)  \rightarrow S_d  \ltimes \left(Aut(T_d) \times \cdots \times Aut(T_d)\right )$
 is as defined in the previous section. 
 This inclusion is denoted by 
$K\succcurlyeq K\times \cdots \times K$. In this case the group is said to be
a regular ((weakly) near) branch group over the subgroup $K$.
\begin{defn}
 Let $G$ be a self-similar group of automorphisms of a $d$-ary rooted tree $d$. $G$ is said to be 
 \textit{self-replicating} if for all $g\in G$ and all $x\in \{0,1,2,\ldots,d-1\}$, there exists 
 an element $h\in St_G(1)$ such that $h_x=g$.
\end{defn}

Regarding the action of $U$ on $T_6$ we have the following:
\begin{thm}\label{U_branch}
 $U$ is a self-replicating weakly near branch group, regular branching over the third commutator subgroup
 $U'''$.
\end{thm}

\begin{proof} Note that $St_{U}(1)$ is generated by the elements $\{b,c,d,aba,aca,ada\}$.
Since we have
$$
\begin{array}{cccc}
b & =  &(a,a,1,b,b,b)  \\
c& = &(a,1,a,c,c,c)  \\
d & =  &(1,a,a,d,d,d) \\
aba & = &(b,b,b,a,a,1)&   \\
aca & = &(c,c,c,a,1,a)&  \\
ada & = &(d,d,d,1,a,a)&  \\
 \end{array}
$$
it follows that $U$ is self-replicating.

We claim that the derived subgroup $U'$ is generated by $(ab)^2,(ac)^2,(ad)^2$.
 From the basic relations  we have  that  $a,b,c,d$ are of order 2 and $b,c,d$ commute with each other.
Hence $U'$ is generated \textit{as a normal subgroup} by $$[a,b]=(ab)^2,[a,c]=(ac)^2,[a,d]=(ad)^2$$
Therefore it is enough to show that the subgroup generated by $(ab)^2,(ac)^2,(ad)^2$ is normal in $U$.
Clearly conjugation by $a$ inverts the elements $(ab)^2,(ac)^2,(ad)^2$.
For other conjugations we have (using the relation $bcd=1$):
$$x (ax)^2 x =(xa)^2=((ax)^2)^{-1}$$
and
$$y(ax)^2 y=(ya)^2 (az)^2=((ay)^2)^{-1} (az)^2 $$
where $x,y,z \in \{b,c,d\}$ are distinct. Therefore $U'$ is generated by $(ab)^2,(ac)^2,(ad)^2$.

Next we claim that $U$ is near weakly branch over the third derived subgroup $U'''$, that is:
 $U''' \succcurlyeq U'''\times U'''\times U'''\times U'''\times U'''\times U'''$.
Let 
$$t=[(ab)^2,(ac)^2], \quad v=[(ab)^2,(ad)^2] \quad w=[(ac)^2,(ad)^2] $$
 $U''$ is generated \textit{as a normal subgroup} by $t,v$ and $w$.
Hence $U''$ is generated by the set $$\{t^{g_1},v^{g_2},w^{g_3} \mid g_i \in U\}$$
It follows that $U'''$ is generated \textit{as a normal subgroup} by the set 
$$S=\{ [t^{g_1},v^{g_2}],[t^{g_3},w^{g_4}] ,[v^{g_5},w^{g_6}] \mid g_i \in U\} $$
We have the following equalities (this can best be checked with the GAP package 

http://www.gap-system.org/Packages/automgrp.html):
$$
\begin{array}{cccc}
 h_1  = & [ [ (ab)^2 ,b ] , [ b , (ca)^2 ] ] & = & (t,*,1,1,1,1)\\
 h_2  = &[ [  (ab)^2 ,b ] , [ c, (da)^2 ] ] & = & (v,1,1,1,*,1) \\

 h_3  = &[ [  c ,(ca)^2 ] , [ b, (da)^2 ] ] & = & (w,1,1,1,1,*) \\

 h_4  = &[ [  b ,(ba)^2 ] , [ d, (ca)^2 ] ] & = & (1,t,1,*,1,1) \\

 h_5  = &[ [  d ,(ad)^2 ] , [ b, (ba)^2 ] ] & = & (1,v,1,1,*,1) \\

 h_6  = &[ [  d ,(ca)^2 ] , [ b, (da)^2 ] ] & = & (1,w,1,1,1,*) \\

 h_7  = &[ [  c ,(ba)^2 ] , [ d, (ca)^2 ] ] & = & (1,1,t,*,1,1) \\

 h_8  = &[ [  d ,(ba)^2 ] , [ c, (da)^2 ] ] & = & (1,1,v,1,*,1) \\

 h_9  = &[ [  c ,(ca)^2 ] , [ d, (da)^2 ] ] & = & (1,1,w,1,1,*) \\
\end{array} 
$$
where $\ast$ are elements of $U$ not of importance.
Clearly $h_i \in U''$ for $i=1,2,3,4,5,6$.
Given $g_1,g_2,g_3,g_4,g_5,g_6 \in U$, due  the fact that $U$ is self-replicating, there are elements
 $\gamma_1,\gamma_2, \gamma_3 ,\gamma_4,
\gamma_5 ,\gamma_6 \in U$ such that 
$$
\begin{array}{c}
 \gamma_1=(g_1,*,*,*,*,*) \\
  \gamma_2=(g_2,*,*,*,*,*) \\
\gamma_3=(g_3,*,*,*,*,*) \\
\gamma_4=(g_4,*,*,*,*,*) \\
\gamma_5=(g_5,*,*,*,*,*) \\
\gamma_6=(g_6,*,*,*,*,*) \\
\end{array}
$$
So, 
$$
\begin{array}{c}
 [h_1^{\gamma_1},h_2^{\gamma_2}]=([t^{g_1},v^{g_2}],1,1,1,1,1) \\

 [h_1^{\gamma_3},h_3^{\gamma_4}]=([t^{g_3},w^{g_4}],1,1,1,1,1) \\

[h_2^{\gamma_5},h_3^{\gamma_6}]=([v^{g_5},w^{g_6}],1,1,1,1,1) \\
\end{array}
$$
and clearly left hand sides are elements of $U'''$.
Using the fact that  $U$ is self-replicating we see that 
$$  U''' \succcurlyeq U'''\times 1 \times 1 \times 1 \times 1 \times 1.$$
Doing same thing in second and third coordinates and using other $h_i$ we see that
$$  U''' \succcurlyeq 1\times U''' \times 1 \times 1 \times 1 \times 1$$
and
$$  U''' \succcurlyeq 1\times 1 \times U''' \times 1 \times 1 \times 1$$
and finally conjugating with $a$ we also have 
$$
\begin{array}{c}
   U''' \succcurlyeq 1\times 1 \times 1 \times U'''\times 1 \times 1 \\
   U''' \succcurlyeq 1\times 1 \times 1 \times 1 \times U'''\times 1 \\
   U''' \succcurlyeq 1\times 1 \times 1 \times 1\times 1 \times U''' \\
\end{array}
$$
which shows that
$$U''' \succcurlyeq U'''\times U'''\times U'''\times U'''\times U'''\times U'''.$$

Clearly $U'''$ is non-trivial since $U$ has non-solvable quotients. 
\end{proof}

Note that $U/U'''$ maps onto the group $\tilde G_{000\ldots}$ and hence is infinite. Also, $U$
cannot have a branch type action (on any rooted tree) since all non-trivial quotients 
of branch groups are virtually abelian, a fact proven in \cite{MR1765119}.

\subsection{Branch structure of general universal groups}\mbox{}

\medskip

In this subsection we will investigate the branch structure of arbitrary universal groups.

For $\om\in \Om$ we have an injection

$$\begin{array}{cccll}
 \phi_\om :  & G_\om & \rightarrow & S_2  \ltimes \left(G_{\tau\om} \times G_{\tau \om}\right ) \\
         &   a      &  \mapsto    & (01)  (1,1) \\
         &  b_\om   &  \mapsto &         (\beta(\om_0 ),b_{\tau\om}) \\
          &  c_\om   &  \mapsto &      (\zeta(\om_0 ),c_{\tau\om}) \\
           &  d_\om   &  \mapsto &    (\delta(\om_0 ),d_{\tau\om}) \\
  \end{array}
$$

For subgroups $H\le G_\om$ and $H\le G_{\tau\om}$ let us write  $K\times K\preceq H$
if $K\times K \le \phi_\om(H)$. Note that this means $H$ contains a subgroup 
isomorphic to $K\times K$.

\begin{prop}\label{third_comm}
 For $\om \in \Om $  we have $ G_{\tau\om}''' \times G_{\tau\om}''' 
 \preceq G_\om '''$
\end{prop} 

\begin{proof}
 Let us assume that $\om_0=0$.   Define 
 $$\pi:U\times U\times U\times U\times U\times U \to G_\om \times G_\om $$
by $\pi(u_1,u_2,u_3,u_4,u_5,u_6)=(\psi_\om(u_1),\psi_\om(u_4))$ where $\psi_\om$ is as defined in section 
4.1. Let $\phi : U\to S_6 \ltimes U\times U\times U\times U\times U\times U$ be the canonical map.

Then the following diagram commutes.

$$\begin{array}[c]{ccc}
St_U(1) &\stackrel{\phi}{\longrightarrow}& U\times U\times U\times U\times U\times U\\
\downarrow\scriptstyle{\psi_\om}&&\downarrow\scriptstyle{\pi}\\
St_{G_\om}(1)&\stackrel{\phi_\om}{\longrightarrow}& G_{\tau\om}\times G_{\tau \om}
\end{array}$$

By Theorem \ref{U_branch}, we have $ U'''\times U'''\times U'''\times U'''\times U'''\times U'''\preceq U'''$.
Since $\psi_\om(U''')=G_\om'''$  we see that $ G_{\tau\om}''' \times G_{\tau\om}''' 
 \preceq G_\om '''$. 
 
 The case when $\om_0=1$ or $\om_0=2$ can be proven similarly by modifying $\pi$.

\end{proof}

\begin{cor}
 For $\om \in \Om_\infty$, $G_\om$ is a branch group.
\end{cor}

\begin{proof}
 It follows by Proposition \ref{third_comm} and an induction argument that for any $n\ge 1$ we have
 $$\prod_{1}^{2^n} G_{\tau^n\om}''' \preceq G_\om'''  $$
It follows that for any $n\ge 1$ , $\prod_{1}^{2^n} G_{\tau^n\om}''' \preceq Rist_{G_\om}(n)$.
Note that for any $\om \in \Om \setminus \Om_0 $, $G_\om'''$ is nontrivial (since $G_\om$ is not solvable)
and also have finite index (since $G_\om$ are just-infinite.) It follows that $Rist_{G_\om}(n)$
has finite index for all $n\ge 1$.
 \end{proof}

For a non-empty subset $\Lambda \subset \Omega$, let us denote the 
universal group corresponding to the family
$\{(G_\om,S_\om) \mid \om \in \Lambda\}$ by $U_\Lambda$.
 Given $\Lambda \subset \Om$ let $T_\Lambda=\displaystyle \bigcup_{\om \in \Lambda}T_\om$ and note that
 $T_\Lambda$ is a (not necessarily regular) subtree of $T_6$. Also note that $T_\Lambda$ is $U$
 invariant (since each $T_\om$ is so) and the restriction of $U$ onto $T_\Lambda$ gives the universal 
 group $U_\Lambda$.
 
 \begin{prop}
  If $\Lambda \subset \Om \setminus \Om_0$  
  then with the action onto $T_\Lambda$, $U_\Lambda$
  is a weakly near branch group.
 \end{prop}
 
 \begin{proof}
  Let $v\in T_\Lambda$ and let $v\in T_\om$ for some $\om\in \Lambda$.
  Let $g$ be a non-trivial element of $Rist_{G_\om}(v)$. Then by the proof of Proposition \ref{third_comm},
  there exists $h\in Rist_U(v)$ such that $\psi_\om(h)=g$. The restriction of $h$ onto
  $T_\Lambda$ gives a non-trivial element in $Rist_{U_\Lambda}(v)$.
 \end{proof}

\section{Universal groups of intermediate growth}

 The aim of this section is to show that there exists a subset $\Lambda \subset \Omega$ of cardinality $2^{\aleph_0}$
such that $U_\Lambda$ has intermediate growth.
This fact was first established in \cite{MR781246}, 
we fix some inaccuracy in the proof of this fact.

First, let us briefly recall basic notions 
related to  the growth of groups. We refer to  \cite{MR1786869,MR2894945,grig_milnor} for a detailed
account on growth and related topics.

Let $G$ be a finitely generated group and $S$ a finite generating set. 
The length of an element (with respect to $S$)
is given by $\ell_S(g)=\min\{n \mid g=s_1s_2\ldots s_n \;,\;i\in S^\pm\}$. The growth function
of $G$ (with respect to $S$) is $\gamma_{G,S}(n)=\#B(G,S,n)$ where
$B(G,S,n)=\{g \in G\mid \ell_S(g)\le n\}$ is the ball of radius $n$.
For two increasing functions $f_1,f_2$ defined on the set of natural numbers, let us write
$f_1 \preceq f_2 $ if there exists $C>0$ such that $f_1(n)\le f_2(Cn)$ for all $n$. Let us also write
$f_1\sim f_2$ if $f_1\preceq f_2$ and $f_2\preceq f_1$, which defines an equivalence relation. 
It can be observed that the
growth functions of a group with respect to different finite generating sets are $\sim$ equivalent and hence
the asymptotic behavior of the growth functions of a group is an invariant of the group.

There are three types of growth for groups: If $\gamma_G \preceq n^d$ for some $d\geq 0$ then
$G$ is said to be of polynomial growth, if $\gamma_G \sim e^n$ then it is said to have
exponential growth. If neither of this happens then the group is said to have
\textit{intermediate growth}. 

If we are talking about the growth of a marked group $(G,S)$, we will simply write $\gamma_G$
for the growth function of $G$ with respect to $S$.

 \begin{lem}\label{diag}
  Let $F=\{(G_i,S_i)\mid i \in I\}\subset \m{M}_k$ be a non-empty subset. Denote by $\gamma_F$ the
  growth function of the diagonal group $(U^{diag}_F,S^{diag}_F)$ of 
  Definition \ref{diaggp}. Then 
  
  \begin{enumerate}
   \item  For all $i\in I$,$\gamma_F(n)\ge \gamma_i(n) \; \text{for all} \; n,$
   \item If $I$ is finite then, $\displaystyle \gamma_F(n)\leq \prod_{i \in I} \gamma_i(n) \; \text{for all}
   \; n.$
  \end{enumerate}

 \end{lem}

 \begin{proof}
 In general, if $(H,K)$ is a marked image of $(G,S)$, then
 $\gamma_{G}(n)\ge \gamma_{H}(n)$ for every $n$. Since all  $(G_i,S_i)$ are marked images
 of the diagonal group, we obtain the first assertion.
 For the second assertion, observe that $B(U^{diag}_F,S^{diag}_F,n)\subset \prod_{i\in I}B(G_i,S_i,n)$.
 \end{proof}

For a positive integer $M$ let $\Om_M\subset \Om_\infty$ be the set of all sequences  for which every subword of length $M$ contains 
all symbols $0,1,2$. 

\begin{thm}\cite[Theorem 3.3]{MR764305} \label{upperbound}
 There exist constants $C$ and $\alpha<1$ depending only on $M$, such that if $\om \in \Om_M$ then
 $$\gamma_{G_\om}(n)\le C^{n^\alpha}\; \text{for all}\; n .$$
\end{thm}
Given natural numbers $r_1,\ldots,r_k$ let 
$$\Lambda_{r_1,\ldots,r_k}=\left\{(012)^{r_1}\eta_1(012)^{r_2}\eta_2\ldots(012)^{r_k}\eta_k(012)^\infty
\mid \eta_i \in \{0,1,2\}\right\}\subset \Om. $$
 where $(012)^\infty$ stands for the periodic sequence $012012012\ldots$.

 For  a sequence of natural numbers $\mathbf{r}=\{r_k\}$, let 
$$\Lambda_{\mathbf{r}}=\left \{(012)^{r_1}\eta_1(012)^{r_2}\eta_2\ldots(012)^{r_k}\eta_k\ldots
\mid \eta_i \in \{0,1,2\}\right \}\subset \Om. $$
Note that both $\Lambda_{r_1,\ldots,r_k}$ and $\Lambda_{\mathbf{r}}$ are subsets of $\Om_4$. Let
 us denote the universal groups $U_{\Lambda_{r_1,\ldots,r_k}}$ and $U_{\Lambda_{\mathbf{r}}}$
 by $U_{r_1,\ldots,r_k}$ and $U_{\mathbf{r}}$  respectively.  Let $\gamma_{r_1,\ldots,r_k}$
 and $\gamma_{\mathbf{r}}$ denote the growth functions (with respect to the canonical generating sets)
 of $U_{r_1,\ldots,r_k}$ and $U_{\mathbf{r}}$ 
 respectively.

 \begin{lem}\label{upper} Given natural numbers $r_1,\ldots,r_k$, there exists a natural number $m$ such
that  $$\gamma_{r_1,\ldots,r_k,x}(m)\leq \left(1+\frac{1}{k}\right)^m \; \text{for any}\; x\in \f{N} .$$
 
\end{lem}

\begin{proof}
 Since $\Lambda_{r_1,\ldots,r_k,x} \subset \Om_4$, by Theorem
 \ref{upperbound} there exists $C$ and $\alpha<1$ (not depending on $x$) such that for all 
 $\om \in \Lambda_{r_1,\ldots,r_k,x} $ we have
 $$\gamma_\om (n)\le C^{n^\alpha} \; \text{for all} \; n.$$
Therefore, by Lemma \ref{diag} (using the fact that $|\Lambda_{r_1,\ldots,r_k,x}|=3^{k+1}$)
we have $$\gamma_{r_1,\ldots,r_k,x}(n)\leq (C^{n^\alpha})^{3^{k+1}}=D^{n^\alpha}\; \text{for all} \; n$$
where $D=C^{3^{k+1}}$ does not depend on $x$. Therefore there exists  a natural number $m$ such that 
$$\gamma_{r_1,\ldots,r_k,x}(m)\leq \left(1+\frac{1}{k}\right)^m \; \text{for any}\; x\in \f{N} .$$
 \end{proof}

  \begin{lem}\cite[Lemma 3]{MR781246}\label{limit} Let $\mathbf{r}=\{r_k\}$ be a sequence  of natural numbers.
 If for some $k$ 
 $$k+r_1+r_2+\ldots+r_k \ge \log_2 2n $$
 then  $\gamma_{r_1,\ldots,r_k}(n)=\gamma_{\mathbf{r}}(n)$.
  \end{lem}
  
  \begin{thm}\cite[Theorem 1]{MR781246} There exists a sequence $\mathbf{r}=\{r_k\}$ such that 
$U_{\mathbf{r}}$
has intermediate growth.
\end{thm}

\begin{proof}
 Let $r_1=1$. By Lemma \ref{upper}, there exists a natural number $n_1$ such that 
 $$\gamma_{r_1,x}(n_1)\leq \left( 1+\frac{1}{1}\right)^{n_1} \: \text{for any}\; x.$$
 Choose $r_2$ such that $2+r_1+r_2 \ge \log_2 2n_1$. Again by Lemma \ref{upper} there
 exists $n_2>n_1$ such that 
  $$\gamma_{r_1,r_2,x}(n_2)\leq \left( 1+\frac{1}{2}\right)^{n_2} \: \text{for any}\; x.$$
  Assume $r_1,\ldots,r_k$ has been already chosen. 
  By Lemma \ref{upper}, there exists $n_{k}>n_{k-1}$ such that
\begin{equation}\label{1}
 \gamma_{r_1,\ldots,r_{k},x}(n_k) \leq \left( 1+\frac{1}{k}\right)^{n_k} \: \text{for any}\; x. 
\end{equation}
Choose $r_{k+1}$ such that 
  \begin{equation}\label{2}
   k+1+r_1+\ldots+r_{k+1}\ge \log_2 2n_k. 
  \end{equation}
Continuing in this manner we construct  sequences $\mathbf{r}=\{r_k\}$ and $\{n_k\}$ for which Equations \ref{1} and \ref{2}
are satisfied. Lemma \ref{limit} and Equation \ref{2} shows that for all $k$ we have 
$$\gamma_{r_1,\ldots,r_{k+1}}(n_k)=\gamma_{\mathbf{r}}(n_k).$$
 Using this and Equation \ref{1} we have,
 $$\lim_{n\rightarrow \infty} \gamma_{\mathbf{r}}(n)^\frac{1}{n}=
 \lim_{k\rightarrow \infty} \gamma_{\mathbf{r}}(n_k)^\frac{1}{n_k}=
 \lim_{k\rightarrow \infty} \gamma_{r_1,\ldots,r_{k+1}}(n_k)^\frac{1}{n_k}
 \leq  \lim_{k\rightarrow \infty}\left( 1+\frac{1}{k}\right)=1 $$
 \end{proof}

\begin{cor}
 There exists a finitely generated group of intermediate growth with $2^{\aleph_0}$ non-isomorphic
 homomorphic images.
\end{cor}

As mentioned in the beginning of this section, this fact was established in \cite{MR781246} with a small inaccuracy.
Our proof mainly follows the lines of \cite{MR781246} only difference being that one needs
Lemma \ref{upper}.

\section{Invariant random subgroups of universal groups}

The aim of this section is to show that there are universal groups with many invariant random subgroups.

\subsection{Preliminaries About Invariant Random Subgroups}\mbox{}

\medskip

Let $G$ be a countable group and let $S(G)$ be the space of subgroups of $G$ endowed with the 
topology having sets of the form $\m{O}_{A,B}=\{N\leq G \mid A\subset N,\;B\cap N=\varnothing\}$
where $A,B$ are finite subsets of $G$, as basis.. $S(G)$ can be identified with a closed subspace
of $\{0,1\}^G$ supplied with by the topology induced from the Tychonoff topology.
The group $G$ acts on $S(G)$ by conjugation and hence forming a topological dynamical system
$(G,S(G))$. We are interested in dynamical system of the form $(G,S(G),\mu)$ where
$\mu$ is a  conjugation invariant  probability measure on $S(G)$.

\begin{defn}
 A conjugation invariant Borel probability measure on $S(G)$ is called an invariant random subgroup (IRS in short).
\end{defn}

The  space  $S(G)$  is a  compact, metrizable,  totally  disconnected   space  which  
(applying  the  Cantor-Bendixon procedure \cite[I.6]{MR1321597})  
consists  of  a  perfect  kernel $\kappa(G)$  and its  complement  $S(G) \setminus \kappa(G)$  which  is  countable.
The  perfect  kernel $\kappa(G)$  is either empty  or  is homeomorphic to a  Cantor set,  and  it  is  empty  if  
and  only  if  $S(G)$  is  countable,  that  is  $G$  has  only  countably many  subgroups.
This  is  the  case, for  instance, for finitely  generated virtually  nilpotent  groups, virtually polycyclic groups,  
some metabelian groups like Baumslag-Solitar groups $B(1,n)$, or Tarski
monsters \cite{MR571100}.

As $\kappa(G)$ is an invariant subset of $S(G)$ with respect to the action of $Aut(G)$
and as the complement $S(G)\setminus \kappa(G)$ is countable, it is clear that a continuous $IRS$ has 
law $\mu$ supported on $\kappa(G)$.

Given a subgroup 
$L\le G$, let $S(G,L)\subset S(G)$ be the set of subgroups containing $L$, which clearly is 
closed. Note that, if $L$ is a normal
subgroup of $G$, then $S(G,L)$ is  invariant under the action of $G$.

Let $\varphi: G \lr H$ be a homomorphism. It induces two  maps
 
$$ 
 \begin{array}{cccc}
  \bar \varphi : &  S(G)  & \lr      & S(H) \\
		 &   N    &  \mapsto & \varphi(N) 
 \end{array}
$$

and 
 
 $$ 
 \begin{array}{cccc}
  \widetilde \varphi : &  S(H)      & \lr      & S(G,Ker(\varphi)) \\
		 &   K  &  \mapsto & \varphi^{-1}(K) 
 \end{array}
$$
 
\begin{lem}\label{lem1}$\mbox{}$
\begin{enumerate}

  \item $\bar \varphi$ is  Borel. 
 \item $\widetilde \varphi $ is continuous.
 \item $\widetilde \varphi(K^{\varphi(g)})=\widetilde \varphi(K)^g$ for all $g\in G$ and $K\leq H$.
 \item $\widetilde \varphi^{-1}(C^g)=\widetilde \varphi^{-1}(C)^{\varphi(g)}$ 
 for all $g\in G$ and $C\subset S(G,Ker(\varphi))$.
 \item If $\varphi$ is surjective, then $\tilde \varphi$ is a homeomorphism.
\end{enumerate}
\end{lem}

\begin{proof}\mbox{}
\begin{enumerate}
\item We claim that $$\bar \varphi ^{-1}(\m{O}_{A,B})=\bigcap_{a\in A}\bigcup_{x\in \varphi^{-1}(a)} 
 \;\bigcap_{y\in \varphi^{-1}(B)}\m{O}_{\{x\},\{y\}}$$
 where $A,B$ are finite subsets of $H$. 
 
 If $\varphi(N)\in \m{O}_{A,B}$ then $A\subset \varphi(N)$ and $B\cap \varphi(N)=\varnothing$.
 This shows that for any $a\in A$ there exists $n_a\in N$ such that $\varphi(n_a)=a$. Also,
 for all $y\in \varphi^{-1}(B)$ we have $y\notin N$. Hence $N$ belongs to the right hand side.
 
 Conversely, let $N\le G$ belong to the right hand side. This means 
 that for all $a \in A$ there exists $n_a\in \varphi^{-1}(a)$
 such that for all $y\in\varphi^{-1}(B)$ we have $N\in \m{O}_{\{n_a\},\{y\}} $.
 For any $a\in A$, we have $\varphi(n_a)=a$ and hence  $A\subset \varphi(N)$.
 Also, if   $ B\cap \varphi(N)$ is nonempty, then the set $N \cap \varphi^{-1}(B)$ is nonempty
 which is not true. Hence $\varphi(N)\in \m{O}_{A,B}$.
 
\smallskip

 Note that in general $\bar \varphi$ is not continuous. For example, the sequence of
 subgroups $(2n+1)\mathbb{Z},n\ge 1$ of $\mathbb{Z}$ converge to the trivial subgroup, but their
 images in $\mathbb{Z}_2$ converge to the whole group.

 \item We claim that  $\widetilde \varphi^{-1}(\m{O}_{C,D})= \m{O}_{\varphi(C),\varphi(D)}$
 where $C,D$ are finite subsets of $G$. In fact, if 
 $\tilde \varphi(K)\in \m{O}_{C,D}$ for some $K\le H$, then $C\subset \varphi^{-1}(K)$ and 
 $D\cap \varphi^{-1}(K)=\varnothing$. It follows that $\varphi(C)\subset K$ and 
 $\varphi(D)\cap K=\varnothing$. This shows that $K\in \m{O}_{\varphi(C),\varphi(D)}$.
 Conversely, if $K\in \m{O}_{\varphi(C),\varphi(D)}$ for some $K\le H$, then $\varphi(C)\subset K$
 and $D\cap \varphi(K)=\varnothing$. It follows that $C\subset \varphi^{-1}(K)$ and 
 $D\cap \varphi^{-1}(K)=\varnothing$ and hence $\tilde \varphi(K)=\varphi^{-1}(K)\in \m{O}_{C,D}$.
 \item This can be verified directly.
 \item This follows from part (2).
 \item If $\varphi$ is surjective, then clearly $\tilde \varphi$ is bijective. 
 Since $S(H)$ is compact, it follows that $\tilde \varphi$ is a homeomorphism,
 
\end{enumerate}

\end{proof}

\begin{cor}\label{cor1}
 If $\mu$ is an $IRS$ of $H$ then the measure $\nu=\widetilde \varphi_\ast (\mu)$ is an $IRS$
 of $G$ supported on the set $\{\varphi^{-1}(K) \mid K\in supp(\mu)\}$. If moreover $\mu$ is continuous, ergodic
 with respect to the action of $H$ and $\varphi$ is surjective, then $\nu$ is continuous and ergodic with respect to the action 
 of $G$. 
\end{cor}

\begin{proof}
 The first part is immediate consequence of Lemma  \ref{lem1} parts $(1)$ and $(3)$. Note
 that the measure $\widetilde \varphi_\ast (\mu)$ is defined on the closed subset $S(G,Ker(\varphi))$
 of $S(G)$,
 and hence can be considered as a measure on $S(G)$ with support in $S(G,Ker(\varphi)$.
 Suppose that $\mu$ is continuous, ergodic and $\varphi$ is surjective. Since $\tilde \varphi$
 is a homeomorphism the measure
 $\nu$ is continuous. Let $C\subset S(G,Ker(\varphi))$ be $G$-invariant. Given $h\in H$,
 pick $g\in G$ such that $\varphi(g)=h$. 
 By Lemma \ref{lem1} part (3),  $\widetilde \varphi ^{-1}(C)^h=\widetilde \varphi ^{-1}(C)^{\varphi(g)}=
 \widetilde \varphi^{-1}(C^g)=\widetilde \varphi^{-1}(C)$. Therefore $\widetilde \varphi^{-1}(C)$ is $H$
 invariant, from which it follows that $\nu(C)=\mu \left( \widetilde \varphi^{-1}(C) \right)\in\{0,1\}$.

\end{proof}



\begin{prop}\label{prop1} Let $X$ be a metrizable Hausdorff topological space and let 
$\mu$ be a Borel measure on $X$.
 Suppose also that a group  $G$ acts on the Borel space $(X,\mu)$ by measure preserving transformations. 
 Then the  map $St: X \lr S(G)$ given by $x\mapsto St_G(x)$  is Borel. Moreover, the measure
 $\nu=St_\ast(\mu)$ is an $IRS$ supported on $\{St_G(x)\mid x\in X\}$.
\end{prop}

\begin{proof}
Observe that the Borel $\sigma$-algebra on $S(G)$ is generated by sets of the
 form $\m{O}_g=\{N\leq G \mid g \in N\}$. Also observe that $St^{-1}(\m{O}_g)=Fix(\varphi_g)$
  where $\varphi_g :X\to X$ given by $\varphi_g(x)=g.x$. Therefore $St^{-1}(\m{O}_g)$  is a  Borel set (see e.g. \cite[Chapter 1]{MR2169627} on sections of Borel maps). 
  This shows that the measure $\nu=St_\ast(\mu)$ is a Borel measure on $S(G)$ 
  with support $\{St_G(x)\mid x\in X\}$. The relation $St_G(g.x)=St_G(x)^{g^{-1}}$ and the $G$
  invariance of $\mu$ show that $\nu$ is conjugation invariant.
 
\end{proof}
It is known (see \cite{1201.3399}) that every $IRS$ of a finitely generated group arises from
a measure preserving action on a Borel probability space $(X,\mu)$.

\medskip

If $T_d$ is the rooted $d$-ary tree, its \emph{boundary} $\partial T_d$ is the set of all infinite rays 
emanating from the root vertex. $\partial T_d$ is in bijection with infinite sequences over the 
alphabet $\{0,1,\ldots,d-1\}$ and hence homeomorphic to a Cantor Set. If $G$ is a group of automorphisms
of a rooted tree $T_d$, its action on $T_d$ extends to an action onto the boundary $\partial T_d$
and this action is by homeomorphisms. Let $\mu$ be the  uniform Bernoulli measure on $\partial T_d$, 
(i.e., the
product of uniform measures on the set $\{0,1,\ldots,d-1\}$). Observe that $\mu$ 
is continuous and invariant under the action
of $Aut(T_d)$ and hence invariant under the action of any subgroup $G\le Aut(T_d)$.
Regarding the the dynamics of such actions the following is known:

\begin{prop}\cite{MR2893544}\label{actiononboundary}
Let $G$ be a countable group of automorphisms of a regular rooted tree $T_d$. Then, the following
are equivalent:

\begin{enumerate}
 \item the group $G$ acts transitively on the levels of $T_d$,
 \item the action of $G$ on $\partial T_d$ is minimal (i.e., orbits are dense),
 \item the action of $G$ on $\partial T_d$ is ergodic with respect to the uniform Bernoulli measure
 on $\partial T_d$.
 \item the action is uniquely ergodic.
\end{enumerate}
\end{prop}

An action of weakly branch type on $T$ gives a totally non-free action on the boundary $\partial T$.

\begin{prop}\cite{MR1899368,MR2893544}\label{injectivestabilizer}
 Let $G\le Aut(T)$ be  weakly branch. Then the map $St: \partial T \to S(G)$ 
given by  $\xi \mapsto St_G(\xi)$ is injective.
\end{prop}

\begin{proof}
 Let $\xi,\eta \in \partial T$ be distinct elements and let $u,v$ be distinct prefixes of $\xi$ and $\eta$ respectively, of same length, say, $n$. Infinite sequences starting with $u$ (respectively with $v$) form a neighbourhood of $\xi$ (respectively of $\eta$) in $\partial T$. We will show that the stabilizer of the neighbourhood of $\eta$ (which is a subgroup in $St_G(\eta)$) is not contained in $St_G(\xi)$.
 
 Let $u$ and $v$ be 
 distinct prefixes of length $n$ of $\xi$ and $\eta$ respectively. 
 Let $g\in Rist_G(u)$ be nontrivial. Since $v$ is not contained in the subtree $T_u$,
 $g$ fixes every infinite sequence starting with $v$. 
  Since $g$ is nontrivial it moves some vertex in $uu_1\in T_u$, say $g(uu_1)=uu_2$ for some
  $u_1\neq u_2$ of lengths $m$. Let $uu'$ be the prefix of $\xi$ of length $n+m$.
  
 If $u'=u$ or $u'=u_2$, then $g(uu')\neq uu'$ and hence $g\notin St_G(\xi)$.
 If both $u'\neq u_1$ and $u'\neq u_2$, by level transitivity let $h\in G$ such that $h(uu_1)=uu'$.
 Then $$(hgh^{-1})(uu')=(hg)(uu_1)=h(uu_2)\neq uu' $$ because $u_1\neq u_2$. Therefore 
 $hgh^{-1}\notin St_G(\xi)$. Since $h(uu_1)=uu_2$, we have $h\in St_G(u)$ and hence
  $hgh^{-1}\in Rist_G(u)$. It follows that $hgh^{-1}\in St^\circ_G(\eta)$.
\end{proof}

As explained in Introduction, this readily provides us with a continuous ergodic $IRS$ on $G$. See for example \cite{DDMN} for a detailed study of this and related measures on the space of Schreier graphs of the Basilica group.

Regarding the action of the Grigorchuk groups $G_\om,\om\in \Om$
on the boundary $\partial T_2$ of the binary tree we obtain the following. 
\begin{prop}\label{ergodicityofGom}
 For $\om \in \Om$ the action of $G_\om$ on $T_2$ is level transitive and hence the action of
 $G_\om$ on $(\partial T_2,\mu)$ is ergodic. Therefore, the induced $IRS$ on $G_\om$ 
 is continuous and ergodic.
\end{prop}

\begin{proof}
 By  Proposition 
 \ref{prop1} the action of $G_\om$ on $(\partial T_2,\mu)$ induces an $IRS$ on $G_\om$.
 This $IRS$ will be continuous by Proposition \ref{injectivestabilizer} and ergodic by
 Proposition \ref{actiononboundary}.
\end{proof}

\subsection{IRS on universal groups} \mbox{}

\medskip

Given $\om_1,\om_2 \in \Om$, let us write $\om_1 \sim \om_2$ if there exists 
$\sigma \in Sym(\{0,1,2\})$ such that $\om_2$ is obtained from $\om_1$ by application of $\sigma$ to each letter
of $\om_1$. Recall that by Theorem \ref{grig84} part (7) we have 
that for $\om_1,\om_2 \in \Om_\infty$, $G_{\om_1}\cong G_{\om_2}$ if and only if $\om_1\sim \om_2$.

For a subset  $\Lambda\subset \Omega$  let $|\Lambda|_{\sim}$ denote the cardinality of
the set of $\sim$ equivalence classes in $\Lambda$.

\begin{prop}
 For $\Lambda \subset \Om_\infty$, $U_\Lambda$ has at least $|\Lambda|_{\sim}$ distinct
 continuous, ergodic  invariant
 random subgroups.
\end{prop}

\begin{proof}
 Fix $\Lambda \subset \Om_\infty$. Let $\varphi_\om : U_\Lambda \lr G_\om$ be the canonical 
 surjection and let $N_\om=Ker(\varphi_\om)$. Note that if $\om\nsim \eta$, then
 by Theorem \ref{grig84} part (7) and the fact that $G_\eta$ is just infinite, 
 we have $N_\eta \nleq N_\om$. For $\om \in \Om$ and $\xi \in \partial T_2$
 let $W_{\om,\xi}=St_{G_\om}(\xi)$. 
 By Proposition \ref{ergodicityofGom}, 
 the canonical action of $G_\om$ onto $(\partial T_2, \mu)$ induces a continuous, ergodic $IRS$ $\mu_\om$ on $G_\om$.
 Moreover, $\mu_\om$ is supported on $\{W_{\om,\xi} \mid \xi \in \partial T_2\}$.

  Let $\nu_\om$ denote the induced $IRS$ on $U_\Lambda$ obtained as described in Corollary \ref{cor1} 
  (i.e., $\nu_\om= (\tilde \varphi_\om)_*(\mu_\om)$). Again by Corollary \ref{cor1}, $\nu_\om$ is continuous and
  ergodic.
  Let $L_{\om,\xi}=\varphi^{-1}_\om(W_{\om,\xi})$
  and note that $\nu_\om$ is supported on $Y_\om=\{L_{\om,\xi}\mid \xi \in \partial T_2\}$.
  Observe that for all $\xi \in \partial T_2$, $L_{\om,\xi}$ contains $N_\om$.
  
  Suppose that for some $\om\nsim\eta \in \Lambda$  and $\xi,\rho\in \partial T_2$ 
  we have $L_{\om,\xi}=L_{\eta,\rho}$. Then $N_\om,N_\eta\leq L_{\om,\xi}$ and hence 
  $L_{\om,\xi}$ contains the subgroup $N= N_\om N_\eta  .$ Since $N_\eta \nleq N_\om$,
  $N$ contains $N_\om$ as a proper subgroup. It follows that the group $U_\Lambda/N$
  is a proper quotient of the group $U_\Lambda / N_\om \cong G_\om$. Since $G_\om$
  is a just infinite group it follows that $N$ and hence $L_{\om,\xi}$ has finite index in 
  $U_\Lambda$. This, in turn shows that $St_{G_\om}(\xi)$ has finite index in $G_\om$
  which is a contradiction.
  Therefore if $\om \nsim \eta$ we see
  that the measures $\nu_\om$ and $\nu_\eta$ have disjoint supports and are in particular
  distinct.
\end{proof}

Combining this with results from Section 5 we obtain the main theorem:

\begin{mainthm}
There is a subset $\Lambda\subset \Omega$ such that the corresponding universal group $U_\Lambda$
has intermediate growth and has $2^{\aleph_0}$ distinct continuous ergodic invariant random subgroups.
\end{mainthm}

\textbf{Acknowledgements:} 

The authors wish to thank Anatoly Vershik and Yaroslav Vorobets for useful discussions and the referee for the careful reading of the manuscript.

\bibliographystyle{alpha}
\bibliography{Universal}

\end{document}